\documentclass[ps]{imsart}

\usepackage[utf8]{inputenc}

\usepackage[left=2.5cm,right=2.5cm,top=2cm,bottom=2cm]{geometry}
\usepackage{color}    
\usepackage{epsfig} 
\usepackage{epsf}            
\usepackage{float}           
\usepackage{amsthm}
\usepackage{graphics}
\usepackage{amsfonts}
\usepackage{amsxtra}
\usepackage{amsmath}
\usepackage{amssymb}
\usepackage{enumerate}
\usepackage{accents}
\usepackage{hyperref}
\hypersetup{colorlinks=false, linkbordercolor={1 1 1}, citebordercolor={1 1 1}, urlbordercolor={1 1 1}} 
\newtheorem{theorem}{Theorem}
\newtheorem{lemma}[theorem]{Lemma}
\newtheorem{corollary}[theorem]{Corollary}

\newtheorem{proposition}[theorem]{Proposition}
\numberwithin{theorem}{section}
\numberwithin{equation}{section}
\numberwithin{figure}{section}
\begin{document}

\begin{frontmatter}

\title{TASEP hydrodynamics using microscopic characteristics}
\runtitle{TASEP hydrodynamics using microscopic characteristics}

\begin{aug}
  \author{Pablo A. Ferrari \thanksref{t2}\ead[label=e1]{pferrari@dm.uba.ar}}

  \address{Universidad de Buenos Aires and IMAS CONICET\\ 
           \printead{e1}}



  \thankstext{t2}{Research partially supported by Mincyt and Mathamsud LSBS-2014.}

  \runauthor{Ferrari}

\end{aug}

\begin{abstract}
The convergence of the totally asymmetric simple exclusion process to the solution of the Burgers equation is a classical result. In his seminal 1981 paper, Herman Rost proved the convergence of the density fields and local equilibrium when the limiting solution of the equation is a rarefaction fan. An important tool of his proof is the subadditive ergodic theorem. We prove his results by showing how second class particles transport the rarefaction-fan solution, as characteristics do for the Burgers equation, avoiding subadditivity. 
Along the way we show laws of large numbers for tagged particles, fluxes and second class particles, and simplify existing proofs in the shock cases. The presentation is self contained. 
\end{abstract}





\tableofcontents

\end{frontmatter}

\newcommand{\utilde}[1]{\underaccent{\tilde}{#1}}
\newcommand{\cA}{\mathcal A}
\newcommand{\cC}{\mathcal C}
\newcommand{\cD}{\mathcal D}
\newcommand{\cE}{\mathcal E}
\newcommand{\cF}{\mathcal F}
\newcommand{\cG}{\mathcal G}
\newcommand{\cH}{\mathcal H}
\newcommand{\cI}{\mathcal I}
\newcommand{\cJ}{\mathcal J}
\newcommand{\cM}{\mathcal M}
\newcommand{\cN}{\mathcal N}
\newcommand{\cP}{\mathcal P}
\newcommand{\cT}{\mathcal T}
\newcommand{\cW}{\mathcal W}
\newcommand{\cX}{\mathcal X}
\newcommand{\cY}{\mathcal Y}
\newcommand{\R}{\mathbb R}
\newcommand{\Z}{\mathbb Z}
\newcommand{\N}{\mathbb N}
\newcommand{\Q}{\mathbb Q}
\newcommand{\vep}{\varepsilon}
\newcommand{\one}{\mathbf 1}
\newcommand{\limt}{\lim_{t\to \infty}}
\newcommand{\lime}{\lim_{\vep\to 0}} 
\newcommand{\nn}{\nonumber}

\section{Introduction}

In the totally asymmetric simple exclusion process (tasep) there is at most a particle per site. Particles jump one unit to the right at rate 1, but jumps to occupied sites are forbidden. Rescaling time and space in the same way, the density of particles converges to a deterministic function which satisfies the Burgers equation. This was first noticed by Rost \cite{MR635270}, who considered an initial configuration with no particles at  positive sites and with particles in each of the remaining sites. He then takes $r$ in $[-1,1]$ and proves that (a) the number of particles at time $t$ to the right of $rt$, divided by $t$ converges almost surely when $t\to\infty$ and (b) the limit coincides with the integral between $r$ and $\infty$ of the solution of the Burgers equation at time 1, with initial condition 1 to the left of the origin and 0 to its right.  This is called convergence of the density fields. Rost also proved that the distribution of particles at time $t$ around the position $rt$ converges as $t$ grows to a product measure whose parameter is the solution of the equation at the space-time point $(r,1)$. This is called local equilibrium because the product measure is invariant for the tasep. These results were then proved for a large family of initial distributions and triggered an impressive set of work on the subject; see Section \ref{notes} later. 

The main novelty of this paper is a new proof of Rost theorem. Rost first uses the subadditive ergodic theorem to prove that the density field converges almost surely and then identifies the limit using couplings with systems of queues in tandem. Our proof shows convergence to the limit in one step, avoiding the use of subadditivity. For each $\rho\in[0,1]$ we couple the process starting with the 1-0 step Rost configuration with a process starting with a stationary product measure at density $\rho$ and show that for each time $t$ the Rost configuration dominates the stationary configuration to the left of $R_t$ and the opposite domination holds to the right of $R_t$; see Lemma \ref{dom1}. Here $R_t$ is a second class particle with respect to the stationary configuration. It is known that $R_t/t$ converges to $(1-2\rho)$ and then the result follows naturally. A colorful and conceptual aspect of the proof is that $1-2\rho$ is the speed of the characteristic of the Burgers equation carrying the density~$\rho$. 

In order to keep the paper self contained we shortly introduce the Burgers equation and the role of characteristics and the graphical construction of the tasep which induces couplings and first and second class particles. We also include a simplified proof of the hydrodynamic limit in the increasing shock case, using second class particles. Along the way we recall the law of large numbers for a tagged particle in equilibrium, which in turn implies law of large numbers for the flux of particles along moving positions and for tagged and isolated second class particles. 

Section \ref{burgers} introduces the Burgers equation and describes the role of characteristics. Section \ref{tasep} gives the graphical construction of the tasep and describes its invariant measures. Section \ref{hl} contains some heuristics for the hydrodynamic limits and states the hydrodynamic limit results. Section \ref{tagged1} contains a proof a the law of large numbers for the tagged particle. Section \ref{coupling1} includes the graphical construction of the coupling and describes the two-class system associated to a coupling of two processes with ordered initial configurations. Section \ref{six} contains the proof of the law of large numbers for the flux and the second class particles. In Section \ref{seven} we prove the hydrodynamic limit for the increasing shock and in  Section \ref{eight} we prove Rost theorem, the hydrodynamics in the the rarefaction fan. Finally Section \ref{notes} includes comments and references.

\section{The Burgers equation}
\label{burgers} 
The one-dimensional Burgers equation is used as a model of transport. The function $u(r,t)\in[0,1]$ represents the density of particles at the space position $r\in \R$ at time $t\in\R^+$.
The density must satisfy
\begin{equation}
  \label{burgers}
  \frac{\partial u}{\partial t} = - \frac{\partial [u(1-u)]}{ \partial r}
\end{equation}
The initial value problem for \eqref{burgers} is to find a solution under the initial condition $u(r,0)=u_0(r),\;r\in\R$,
where  $u_0:\R\to[0,1]$ is given. 
In this note we only consider the following family of initial conditions:
\begin{align}
  \label{url}
u_0(r) = u^{\lambda,\rho}(r):=
\begin{cases}
  \lambda&\hbox{if }r\le 0\\
 \rho&\hbox{if }r> 0 
\end{cases}
\end{align}
where $\rho,\lambda\in[0,1]$.
Lax \cite{MR0350216} explains how to treat this case.
Differentiating \eqref{burgers} we get
\begin{equation}
\frac{\partial u }{ \partial t}=  -(1-2u) \frac{\partial u  }{ \partial r} 
\end{equation}
so that $u$ is constant along $w(t)$ with $w(0)=r$, the trajectory satisfying
$\frac{d}{dt}w= (1-2u)$.
That is, $u$ propagates with
speed $(1-2u)$: $u(w(t),t)= u_0(w(0))$. 
These trajectories are called \emph{characteristics}. 
If different characteristics meet, carrying two different solutions to
the same point, then the solution has a shock or discontinuity at that position.
In our case the discontinuity is present in the initial
condition.  The cases $\lambda<\rho$ and $\lambda>\rho$ are qualitative different. 

\paragraph{Shock case}  When $\lambda<\rho$ the characteristics starting at $r>0$ and $-r$ have
speed $(1-2\rho)$ and $(1-2\lambda)$ respectively and meet at time $t(r)=
r/(\rho-\lambda)$ at position $(1-\lambda-\rho)r/(\rho-\lambda)$. 
\begin{figure}[h]
\centering
\includegraphics[
clip=true, 
scale=.8
]
{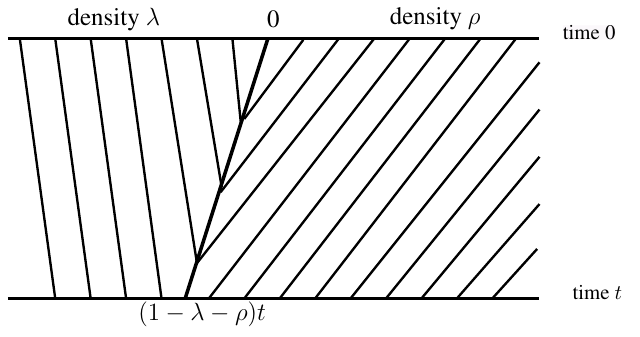} 
\caption{Shocks and characteristics in the Burgers equation. The characteristics starting at
$r$ and $-r$ that go at velocity $1-2\rho$ and $1-2\lambda$ respectively with $\rho>\lambda$. The
center line is the shock that travels at velocity $1-\rho-\lambda$.}
\label{fig2}
\end{figure}%
Take $a<b$
large enough to guarantee that the shock is inside $[a,b]$ for times in $[0,t]$. By conservation of mass:
\begin{align}
  \label{amb}
\frac{d}{dt} \int_a^b u(r,t)\,dr\, =\, u(a,t)(1-u(a,t))\,-\,u(b,t)(1-u(b,t))
\end{align}
Since $\int_a^b u(r,t)\,dr= \lambda (y_t-a)+\rho (b-y_t)$, where $y_t$ is the position of the shock at time $t$, we have
\[
y'_t (\lambda-\rho) = \lambda(1-\lambda) -\rho(1-\rho) 
\]
and $y_t = (1-\lambda-\rho)t$.
We conclude that for $\lambda<\rho$, the solution of the initial value problem $u(r,t)$ is $\rho$
for $r>vt$ and $\lambda$ for $r<vt$, that is,
\[
u(r,t)= u^{\lambda,\rho}(r-vt).
\]

\paragraph{The rarefaction fan} When $\lambda>\rho$ the characteristics emanating at the left of the origin have speed $(1-2\lambda)<(1-2\rho)$, the speed to the right
and there is a family of characteristics emanating from the origin with speeds $(1-2\alpha)$ for $\lambda\ge\alpha\ge\rho$. 
\begin{figure}[h]
\centering
\includegraphics[
clip=true, 
width=7cm
]
{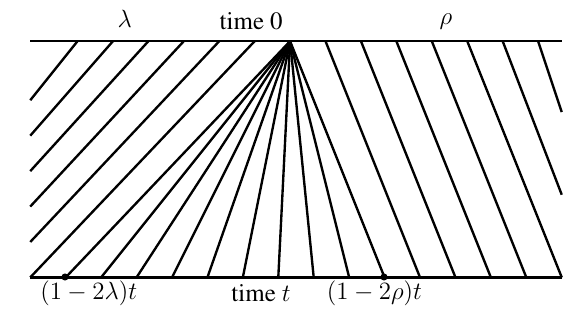} 
\caption{The rarefaction fan. Here $\lambda>\rho$.}
\label{fig3}
\end{figure}%
The solution is then
\begin{align}
  \label{rarefaction-solution}
u(r,t) =
\begin{cases}
  \lambda & \hbox{if } r<(1-2\lambda)t\\[2mm]
\displaystyle{\frac{t-r}{2t}}  & \hbox{if } (1-2\lambda)t\le r\le (1-2\rho)t\\[2mm]
\rho &\hbox{if } r>(1-2\rho)t
\end{cases}
\end{align}
The characteristic starting at the origin with speed $(1-2\alpha)$ carries the solution $\alpha$:
\begin{align}
  u\big((1-2\alpha)t,t\big) = \alpha,\qquad \lambda\ge \alpha\ge \rho.
\end{align}

The above solution is a \emph{weak solution}, that is, for all   
$\Phi\in C_0^\infty$ with compact support, 
\begin{equation}
\int\int \left(\frac{\partial\Phi }{ \partial t} u + \frac{\partial\Phi }{
\partial r}u(1-u) \right) dr dt = 0.
\end{equation}
The solution may be not unique, but \eqref{rarefaction-solution} comes
as a limit when $\beta \to 0$ of the unique solution of the
(viscid) Burgers equation
\begin{equation}
\label{2.2}
\frac{\partial u }{ \partial t} 
=- \frac{\partial [u(1-u)]}{ \partial r}+ \beta \frac{\partial^2 u }{ \partial r^2}.
\end{equation} 
This solution, called \emph{entropic}, is selected by the hydrodynamic limit of the tasep, as we will see.

\section{The tasep}
\label{tasep}

We construct now the tasep.
Call \emph{sites} the elements of $\Z$ and \emph{configurations} the elements of the space $\{0,1\}^\Z$, endowed with the product topology. When $\eta(x)=1$ we say that $\eta$ has a \emph{particle} at site $x$, otherwise there is a \emph{hole}.  


\paragraph{Harris graphical construction} We define directly the graphical construction of the process, a method due to Harris \cite{MR0488377}. The process in $\{0,1\}^\Z$ is given as a function of an initial configuration $\eta$ and a Poisson process $\omega$ on $\Z\times\R^+$ with rate 1; $\omega$ is a random discrete subset of $\Z\times\R$. When $(x,t)\in\omega$ we say that there is an arrow $x\to x+1$ at time $t$. 
Fix a time $T>0$. 
\begin{figure}[h]
\centering
\includegraphics[
clip=true, 
scale=.8
]
{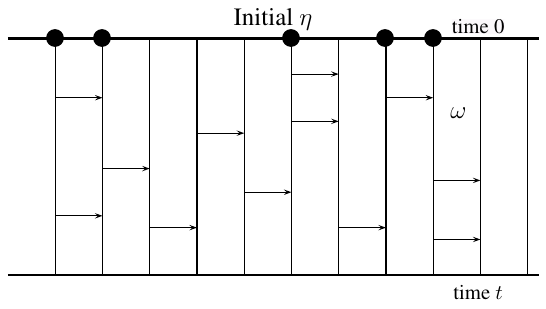} 
\caption{A typical $\omega$, represented by arrows and the initial configuration $\eta$, where  particles are represented by dots. }
\label{fig4}
\end{figure}%
For almost all $\omega$  there is a double infinite sequence of sites $x_i=x_i(\omega)$, $i\in \Z$ with no arrows $x_i\to x_i+1$ in $(0,T)$. The space $\Z$ is then partitioned into finite boxes
$[x_i+1,x_{i+1}]\cap \Z$ with no arrows connecting boxes in the time interval $[0,T]$. Take $\omega$ satisfying this property and an arbitrary initial configuration $\eta$ and construct $\eta_t$, $0\le t\le T$, as a function of $\eta$ and $\omega$, as follows. 

Since the boxes are finite, we can label the arrows inside each box by order of
appearance. Take a box. If the first arrow in the box is $(x,t)$ and at time $t-$ there is a
particle at $x$ and no particle at $x+1$, then the particle follows the arrow $x\to x+1$ so that at time $t$ there is a particle at $x+1$ and no particle at
$x$.
If before the arrow from $x$ to $x+1$ there is a different event (two particles, two holes or a particle at $x+1$ and no particle at $x$), then nothing happens: the configuration after the arrow is exactly the same as before.
Repeat the procedure for the following arrows until the last arrow in the box. Proceed to next box and obtain a particle configuration depending on the initial $\eta$ and the Poisson realization $\omega$, denoted $\eta_{t}[\eta,\omega]$, $0\le t\le T$.
\begin{figure}[h]
\centering
\includegraphics[
clip=true, 
scale=.8
]
{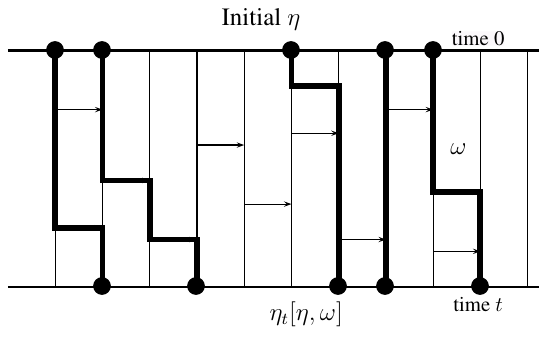} 
\caption{A typical construction. Particles follow arrows when destination site is empty.}
\label{fig5}
\end{figure}%
For times greater than $T$, use $\eta_{T}$ as initial configuration and repeat the procedure to construct
the process between $T$ and $2 T$, using the arrows of $\omega$ with times in $[T,2T]$ and so on. In this way we have constructed the process 
\[
(\eta_t[\eta,\omega]:t\ge0).
\]
The process satisfies the almost sure Markov property
\begin{align}
  \label{asmp}
\eta_{t+s}[\eta,\omega]= \eta_s\big[\eta_t[\eta,\omega],\tau_t\omega\big],
\end{align}
where $\tau_t\omega :=  \{(x,s): (x,t+s)\in\omega\}$ has the same distribution as $\omega$ and it is independent of $\omega\cap(\Z\times[0,t])$, by the properties of the Poisson process $\omega$.
This implies that the process $\eta_t$ is Markov. 
Usually we omit the dependence on $\omega$ in the notation.

\paragraph{\bf Product measures}  
Let 
\begin{align}
\label{uuu}
  U=(U(x):x\in\Z) := \;\hbox{iid random variables uniformly distributed in }[0,1].
\end{align}
Assume that $U$ is independent of $\omega$. 
For each $\rho\in[0,1]$ define $\eta^\rho=\eta^\rho[U]$ by
\begin{align}
\label{etalambda}
  \eta^\rho(x) := \one \{U(x)<\rho\}.
\end{align}
where $\one B$ is the indicator function of  $B$. 
All configurations in this paper defined in function of $U$ are naturally coupled by using the \emph{same} uniform random variables \eqref{uuu}; we drop the dependency of $U$ to lighten the notation. 
The distribution of $\eta^\rho$ is a Bernoulli product measure.
Define 
\begin{align}
  \label{fa}
f_A(\eta):= \prod_{x\in A} \eta(x).
\end{align}
If $\zeta$ is a random configuration in $\{0,1\}^\Z$, then $(Ef_A(\zeta):A\subset \Z,$ finite$)$ characterizes the distribution of $\zeta$. In particular, the distribution of $\eta^\rho$ is characterized by $E f_A(\eta^\rho)= \rho^{|A|}$, where $|A|$ is the cardinal of $A$.

Denote 
\begin{align}
  \label{etarhot}
\eta^\rho_t := \eta_t[\eta^\rho,\omega]
\end{align}
The configuration $\eta^\rho_t$ is a function of $U$ and $\omega$. 
We denote  $P$ and $E$ the probability and expectation associated to the probability space induced by the independent random elements $U$ and $\omega$.

\begin{lemma}
\label{nu-rho-invariant}
For each $\rho\in[0,1]$,  the distribution of $\eta^\rho$ 
is invariant for the tasep. That is, for any finite $A\subset \Z$ we have 
$E(f_A(\eta^\rho_t)) = 
\rho^{|A|}$, for all $t\ge0$.
\end{lemma}
This lemma is proved in Liggett \cite{MR0418291}. The configurations $\zeta^{(n)}(x) := \one\{x\ge n\}$ are frozen because all particles are blocked. In the same paper Liggett shows that all the invariant measures are combination of the Bernoulli product measures and the blocking measures, those concentrating mass on the frozen configurations $\eta^{(n)}$.

\section{The hydrodynamic limit}
\label{hl}

\paragraph{Heuristic derivation of Burgers equation from tasep} 
Using the forwards Kolmogorov equation for the function $f(\eta) = \eta(x)$ we get
\begin{align}
\label{de1}
  \frac{d}{dt} E(\eta_t(x)) = E\big[-\eta_t(x)(1-\eta_t(x+1))+ \eta_t(x-1)(1-\eta_t(x))\big], 
\end{align}
Fix an $\vep>0$ which will go later to zero and define 
\[
u^\vep(r,t):=E [\eta_{\vep^{-1}t}(\vep^{-1}r)],
\]
where $\vep^{-1}r$ is an abuse of notation for integer part of
$\vep^{-1}r$. 
Putting the $\vep$'s in \eqref{de1} we get
\begin{align}
\frac{ d}{ dt} u^\vep(r,t))
=\vep^{-1}E\bigl[&
-\eta_{t\vep^{-1}}(r\vep^{-1})(1-\eta_{t\vep^{-1}}(r\vep^{-1}+1))\nonumber\\
&\qquad +\,\eta_{t\vep^{-1}}(r\vep^{-1}-1)\,(1-\eta_{t\vep^{-1}}(r\vep^{-1}))\bigr].\label{de2}
\end{align}
Assume that there exist a limit 
\[
u(r,t):=\lime u^\vep(r,t) 
\]
and that the distribution of $\eta_{\vep^{-1}t}$ around $\vep^{-1}r$ is approximately product, that is,
\[
\lime E \big[\eta_{t\vep^{-1}}(r\vep^{-1})\,\eta_{t\vep^{-1}}(r\vep^{-1}+1)\big]\;=\; (u(r,t))^2.
\]
Assume further that $u(r,t)$ is differentiable in $r$. In this case, the right hand side of \eqref{de2} must converge to minus the derivative of $u(r,t)(1-u(r,t))$, that is, the limiting $u(r,t)$ must satisfy the Burgers equation. This heuristic argument may also be a script of a proof of the convergence of the tasep density to a solution of the Burgers equation. Instead, we show directly the convergence in the terms described by \eqref{aib} and \eqref{local-equilibrium} later.

\paragraph{Hydrodynamics limit. General case} Consider the Burgers equation with initial data $u_0$ such that there exists a unique entropic  weak solution $u(r,t)$ for the initial value problem~\eqref{burgers}-\eqref{url}. Take the uniform random variables $U$ defined in \eqref{uuu} and define
\begin{align}
\label{uue}
  \zeta^\vep(x) := \one \{U(x)\le u_0(\vep x)\}.
\end{align}
That is, for each $\vep>0$, the random configuration $\zeta^\vep$ is a sequence of independent Bernoulli random variables with varying parameter induced by $u_0$ for the mesh $\vep$. 
Let $\zeta^\vep_t$ be the tasep with random initial configuration $\zeta^\vep$:
\begin{align}
  \label{eut}
\zeta^\vep_t:= \eta_t[\zeta^\vep,\omega].
\end{align}
Denote  $\tau_z$ the translation operator by $z$, defined by $(\tau_z\eta)(x) = \eta(\lfloor x+z\rfloor)$, here $\lfloor z\rfloor$ is the integer part of $z$.

\begin{theorem}  [Hydrodynamic limits by several authors]
\label{density-fields}

Let $u(r,t)$ be the solution of the Burgers equation with initial condition $u_0$. Let $\zeta^\vep$ be given by \eqref{uue} and $\zeta^\vep_t$ be the tasep with initial condition $\zeta^\vep$ defined in \eqref{eut}. Then, 

\noindent\emph{Convergence of the  density fields. } 
For all real numbers $a<b$ and for all $t\ge 0$, 
\begin{align}
  \label{aib}
\lime \vep \sum_{x: a\le \vep x\le b}   \zeta^\vep_{\vep^{-1}t}(x)   = \int_a^b  u(r,t) dr,\qquad \hbox{a.s.}
\end{align}
\emph{Local-equilibrium. } At the continuity points of $u(r,t)$,
\begin{align}
\label{local-equilibrium}
\lime E[f_{A}(\tau_{\vep^{-1}r}\zeta^\vep_{\vep^{-1}t})] 
\;= u(r,t)^{|A|}.
\end{align}

\end{theorem}
The limit \eqref{local-equilibrium} gives weak convergence of the particle distribution at the points of
continuity of $u(r,t)$ to the distribution of $\eta^{u(r,t)}$, which is an invariant measure. When $A=\{0\}$, the limit \eqref{local-equilibrium} is the so called \emph{density profile}:
\begin{align}
  \label{4.1a}
\lime E[\zeta^\vep_{\vep^{-1}t}(\vep^{-1}r)] \;= \; u(r,t),
\end{align}
ignoring the integer parts, as abuse of notation. In Section \ref{notes} we give references to the proof of this Theorem.

\paragraph{Hydrodynamic limit. Shock case} Consider the case corresponding to $u_0= u^{\lambda,\rho}$ and $t=1$. 
Let $\lambda,\rho\in[0,1]$ and $\eta^{\lambda,\rho}=\eta^{\lambda,\rho}[U]$ be defined by
\begin{align}
  \label{etarl1}
\eta^{\lambda,\rho}(x) :=
\begin{cases}
  \one\{U(x) \le \lambda,\} &\hbox{if } x\le0\\
\one\{U(x) \le \rho,\} &\hbox{if } x>0.
\end{cases}
\end{align}
where $U(x)$ defined in \eqref{uuu} are the same we used to define $\eta^\rho$. As before we denote
\begin{align}
  \eta^{\lambda,\rho}_t:= \eta_t[\eta^{\lambda,\rho},\omega], \nn
\end{align}
 a function of $U$ and $\omega$. In the rest of the paper we fix macroscopic time equal to 1 and use $t$ as scaling parameter.

We prove the following theorem. The result is a particular case of Theorem \ref{density-fields}, which is known but the methods are new for the rarefaction case.
\begin{theorem}
  \label{t4.2}
  For all real numbers $a<b$, 
\begin{align}
  \label{aib1}
\limt \frac1t \sum_{x: at\le  x\le bt}   \eta^{\lambda,\rho}_t(x)   = \int_a^b  u^{\lambda,\rho}(r,1) dr,\qquad \hbox{a.s.}
\end{align}
 At the continuity points of $u^{\lambda,\rho}(\cdot,1)$, we have
\begin{align}
\label{local-equilibrium1}
\limt E[f_{A}(\tau_{tr}\eta^{\lambda,\rho}_{t})] 
\;= u(r,1)^{|A|}.
\end{align}
\end{theorem}

\paragraph{Sketch of proof of Theorem \ref{t4.2}} 
The proofs are based on the coupling of the tasep obtained by using the same $U$ and $\omega$ for \emph{all} initial conditions. 
A crucial property of the coupling is attractivity, meaning that initial coordinate-wise ordered configurations keep their order under the coupled evolution. In turn, attractivity permits to describe the system in terms of first and second class particles, a tool largely used in the literature. During the proof we will prove laws of large numbers for (a) a tagged particle for the stationary process $\eta^\lambda_t$, (b) the flux of $\eta^\lambda_t$ particles along a traveler with constant speed, (c) a second class particle for the process with initial shock configuration $\eta^{\lambda,\rho}$ with $\lambda<\rho$ and (d) a second class particle for the stationary process $\eta^\lambda_t$. The main novelty is the microscopic counterpart of Figure \ref{fig3}.

\section{The tagged particle}
\label{tagged1}
Take a configuration $\eta$ with infinitely many particles to the left and right of the origin and tag its particles as follows:
\begin{align}
  X(i)[\eta]:= 
  \begin{cases}
 \max\{x\le 0: \eta(x)=1\}&\hbox{if } i=0\\
\min\{x>X(i-1):\eta(x)=1\}&\hbox{if } i>0\\
\max\{x<X(i+1): \eta(x)=1\}&\hbox{if } i<0.
  \end{cases}
\end{align}
We are interested in configurations with a particle at the origin. So, define
\begin{align}
  \label{tilde-eta}
\tilde\eta(x) :=
\begin{cases}
  1 &\hbox{ if }x=0\\
\eta(x)&\hbox{ otherwise}
\end{cases};
\qquad\qquad \tilde\eta_t := \eta_t[\tilde\eta,\omega].
\end{align}
The positions of the particles at time $t$ can be recovered from the graphical construction by following the thick trajectories, see Figure \ref{tagged-particles}. Call $X_t(i)[\tilde\eta,\omega]$ the position of the $i$-th particle at time $t$; when $\eta$ and $\omega$ are understood we just denote $X_t(i)$. Call $X_t:=X_t(0)$ the position of the tagged particle initially at the origin and define the process as seen from that tagged particle by
\begin{align}
  \label{tpp}
\tau_{X_t}\eta_t[\tilde\eta,\omega].
\end{align}
\begin{figure}[h]
\centering
\includegraphics[
clip=true, 
scale=.8
]
{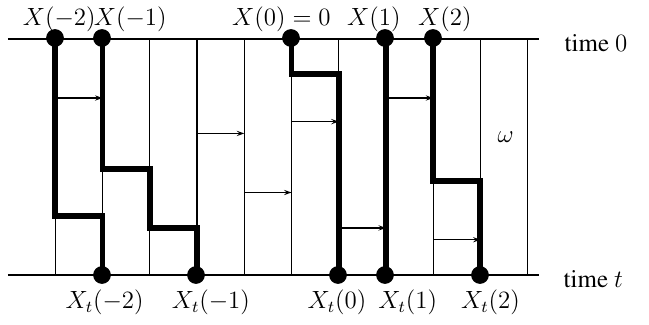} 
\caption{Trajectories of the tagged particles.}
\label{tagged-particles}
\end{figure}%
Add a particle to the configuration $\eta^\rho$ as in \eqref{tilde-eta} to get $\tilde\eta^\rho$. The law of $\tilde\eta^\rho$ is the Bernoulli product measure conditioned to have a particle at the origin.  The distribution of $\tilde\eta^\rho$ is invariant for the process as seen from the tagged particle:  $\tau_{X_t}\tilde\eta^\rho_t$ has the same distribution as $\tilde\eta^{\rho}$ for all $t\ge 0$,   see \cite{MR866349}, for instance. This invariance is crucial in the two alternative proofs of the law of large numbers of the next proposition but it is not necessary for the rest of the arguments of this paper.

\begin{proposition}[Law of large numbers for the tagged particle]
  \label{lln-tagged}
Let $X_t$ be the position of the tagged particle initially at the origin for the process with random initial configuration $\tilde\eta^\rho$. Then,
\begin{equation}
\label{lln-burke}
\limt\frac{X_{t}}t=(1-\rho), \quad \hbox{a.s.}
\end{equation}
\end{proposition}

\begin{proof}[Sketch proof]
A proof based in Burke's theorem \cite{MR0083416} goes as follows. Think that the particles are servers and the holes are customers of a system of infinitely many queues in series so that $X_t(i)$ is the position of server $i$ at time $t$, $i\in\Z$ with $X_t(0)=X_t$. Let the block of successive holes to the right of $X_t(i)$ be the queue of server $i$ at time $t$. Each time server-$i$ jumps to the right, a customer is served and goes to the queue of server-$(i-1)$. Burke's theorem says that if the initial random configuration is $\tilde\eta^\rho$, then the process $(X_t,t\ge 0)$ is a Poisson process of rate $(1-\rho)$. This fact was observed by Kesten in Example 3.2 of the historical Spitzer's 1970 paper \cite{MR0268959}; see  \cite{MR832016} or \cite{MR1385350} for proofs in this context. As a corollary we get the law of large numbers \eqref{lln-burke}.

Alternatively, Saada \cite{MR877609} proves that the process  $(\tau_{X_t}\tilde\eta^\rho:t\ge0)$ is ergodic, which in turn implies the law of large numbers; this argument avoids the use of Burke's theorem.
\end{proof}

\section{Coupling and two-class tasep}
\label{coupling1}

The graphical construction provides a natural coupling of the tasep starting with two or more different configurations. Let $\eta,\eta'$ be initial configurations and define the coupling
\[
\big((\eta_t,\eta'_t): t\ge0\big):=\big((\eta_t[\eta,\omega], \eta_t[\eta',\omega]): t\ge0\big).
\]
\begin{figure}[h]
\centering
\includegraphics[
clip=true, 
scale=.7
]
{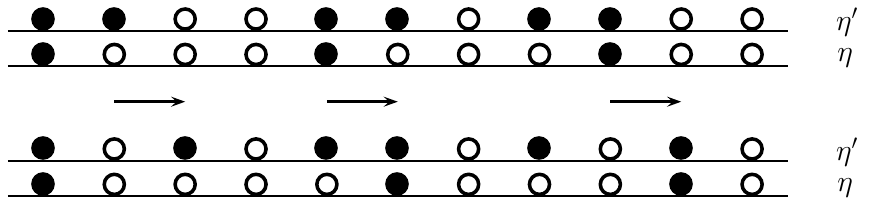} 
\caption{Coupling. Configurations  $\eta'$ and $\eta$ before and after
3 possible arrows. }
\label{coupling}
\end{figure}%
This amounts to use the same arrows for both marginals. By construction, each marginal of the coupling has the distribution of the tasep. Particles at site $x$ of each marginal try to jump at the same time, but the jump occurs only if the destination site $x+1$ is empty in the corresponding marginal. 

Denote $\eta\le \eta'$ if $\eta(x)\le\eta'(x)$ for all $x\in\Z$. 

\begin{lemma}
  \label{attractivity1}
Attractivity. For all $t\ge 0$ we have 
\begin{align}
\label{attractivity}
  \eta\le\eta'\quad \hbox{implies}\quad \eta_{t}\le\eta'_{t} \qquad \hbox{a.s.}
\end{align}
Discrepancy conservation. If $\eta\le\eta'$, and the number of discrepancies is finite, then
\begin{align}
\label{discrepancy-conservation}
  \sum_x(\eta'(x)-\eta(x)) = \sum_x(\eta'_t(x)-\eta_t(x)).
\end{align}\end{lemma}

\begin{proof} To show \eqref{attractivity} it is sufficient to check that if $\eta_{t-}\le \eta'_{t-}$ and $(t,x)\in\omega$, that is, there is an arrow from $x$ to $x+1$ at time $t$, then $\eta_{t}\le \eta'_{t}$ , that is, the domination still holds after the arrow. The same exploration shows that the number of discrepancies does not change after the arrow.
\end{proof}

\paragraph{First and second class particles} Fix $\eta\le \eta'$ and call
\begin{align}
  \sigma_t:= \eta_t[\eta,\omega],\qquad \xi_t := \eta_t[\eta',\omega]-\eta_t[\eta,\omega].
\end{align}
By definition $\sigma_t\in\{0,1\}^\Z$ and by attractivity, $\xi_t\in\{0,1\}^\Z$.  
We call \emph{first class} the $\sigma$ particles and \emph{second class} the $\xi$ particles.
\begin{figure}[h]
\centering
\includegraphics[
clip=true, 
scale=.7
]
{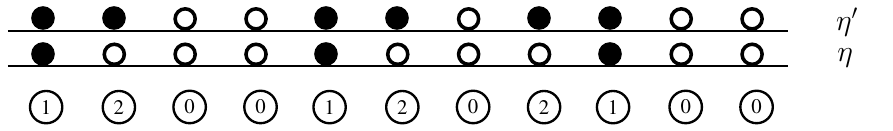} 
\caption{ The $(\sigma,\xi)$ configuration associated to $(\eta,\eta')$ of figure \ref{coupling}. $\sigma$ particles are labeled $1$, $\xi$ particles are labeled $2$ and holes are labeled $0$.}
\label{f-first-second}
\end{figure}%
The process $((\sigma_t,\xi_t):t\ge0)$ is Markov; it can be constructed directly as function of $\omega$ and the initial configurations $\sigma$ and $\xi$, as follows. At each site there is at most one particle, either $\xi$ or $\sigma$. Arrows involving $\xi$-$\xi$, $\sigma$-$\sigma$, $\xi$-0, $\sigma$-0 particles, use the same rules as the tasep, but arrows involving $\sigma$-$\xi$ particles follow the rules (a) if $\sigma \to \xi$ then  the particles interchange positions and (b) if $\xi\to\sigma$, then nothing happens. In other words, $\xi$ particles behave as particles when interacting with holes and as holes when interacting with $\sigma$ particles. 
\begin{figure}[h]
\centering
\includegraphics[
clip=true, 
scale=.7
]
{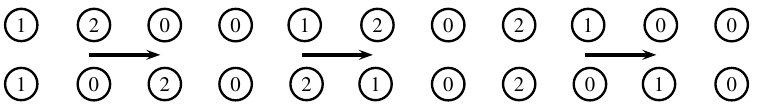} 
\caption{Another way of looking at the coupling. We see three possible jumps of first and second class particles associated to the configuration $\eta'$ and $\eta$ of figure \ref{coupling}. The upper line shows the configuration before the jumps and the lower line the one after the jumps.}
\label{first-second-jumps}
\end{figure}%

The vector $(\sigma_t,\xi_t)$ depends on the initial configuration $(\sigma,\xi)=(\eta,\eta'-\eta)$ and on~$\omega$. When this needs to be stressed we denote 
\begin{align}
  \label{sixiom}
(\sigma_t,\xi_t)=(\sigma_t,\xi_t)[(\sigma,\xi),\omega)]=(\sigma_t[(\sigma,\xi),\omega)],\xi_t[(\sigma,\xi),\omega)]),
\end{align}
either way.

\section{Law of large numbers}
\label{six}

\paragraph{Flux}
Let $(y_t:t\ge 0)$ be an arbitrary trajectory in $\R$ with $y(0)=0$. Define the \emph{flux} of particles along $y_t$ by 
\begin{align}
  &F_{y_t}(t)[\eta,\omega]
:= \sum_{i\le 0} \one\{X_t(i)[\eta,\omega]>y_t\} -\sum_{i>0}\one\{X_t(i)[\eta,\omega]\le y_t\}
\end{align}
\begin{figure}[h]
\centering
\includegraphics[
clip=true, 
scale=.7
]
{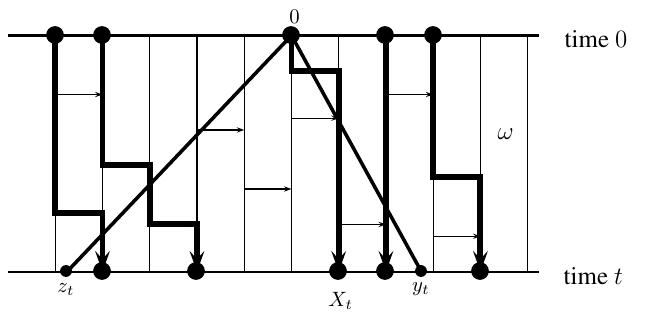} 
\caption{The flux along trajectory $y_t$ is $-1$ and the flux along trajectory $z_t$ is 3. }
\label{d-flux}
\end{figure}%
Consider the configuration $\tilde\eta$ defined from $\eta$ in \eqref{tilde-eta}, having  a particle at the origin. Recall $X_t$ is the position of the tagged particle of $\tilde\eta$ initially at the origin 
and observe that due to the exclusion interaction and the nearest neighbor jumps, the flux of $\tilde\eta$ particles along the tagged particle $X_t$ is null:
\begin{align}
  \label{ftxt}
F_{X_t}(t)[\tilde\eta,\omega]\equiv 0.
\end{align}
Hence we have the following alternative expression for the flux of $\tilde\eta$ particles. 
\begin{align}
  \label{fluxt}
F_{y_t}(t)[\tilde\eta,\omega] = \sum_x \tilde\eta_t(x)\,\big( \one\{y_t< x\le X_t\}-\one\{X_t< x\le y_t\}\big); 
\end{align}
only one of the indicator functions is no null in each term of \eqref{flux}. And, since $\eta$ and $\tilde\eta$ have at most one discrepancy which is conserved by \eqref{discrepancy-conservation}, 
\begin{align}
  \label{flux}
F_{y_t}(t)[\eta,\omega] = \sum_x \eta_t(x)\,\big( \one\{y_t< x\le X_t\}-\one\{X_t< x\le y_t\}\big)\,+\, O(1),
\end{align}
where $O(1)$ is some function of $U$, $\omega$ and $t$ satisfying that  $|O(1)|\le$ Constant. The function may change from line to line, but in any case $O(1)/t$ goes to zero almost surely when $t\to\infty$.

\begin{proposition} Let $a\in \R$. Then, 
\begin{align}
\label{fl1}
\limt \frac{F_{at}(t)[\eta^\rho,\omega]}{ t} &= \rho[(1-\rho)-a],\qquad  \hbox{ a.s.}
\end{align}
\end{proposition}

\begin{proof} Using \eqref{flux} we can write
\begin{align}
  F_{at}(t)[\eta^\rho,\omega] & = \sum_x \eta^\rho_t(x)\,\big( \one\{at< x\le (1-\rho)t\}-\one\{ (1-\rho)t< x\le at\}\big)\nonumber\\
&\quad + \sum_x \eta^\rho_t(x)\,\big(\one\{ (1-\rho)t< x\le X_t\}-\one\{X_t< x\le (1-\rho)t\}\big)+O(1). \nonumber
\end{align}
Dividing by $t$ and taking $t\to\infty$, the first term converges a.s.\/ to $\rho[(1-\rho)-a]$ because  $\eta^\rho_t$ is a sequence of iid Bernoulli$(\rho)$ random variables by Lemma \ref{nu-rho-invariant}. 
The absolute value of the second term is bounded by $|X_t-(1-\rho)t|/t$ which goes to zero a.s.\/ by Proposition~\ref{lln-tagged}. 
\end{proof}

\paragraph{Tagged second class particle} 
Take $0\le\lambda<\rho\le1$ and using always the same $U$ and $\omega$ define the two-class process
\begin{align}
  \label{rlt}
(\sigma_t,\xi_t):= (\eta^{\lambda}_t,\eta^{\rho}_t-\eta^\lambda_t).
\end{align}
The marginal laws of $\sigma_t$ and $\sigma_t+\xi_t$ are stationary but the process $(\sigma_t,\xi_t)$ is not stationary. 
Take off a particle of $\eta$ at the origin defining $\utilde\eta$ as the configuration
\begin{align}
  \label{utilde}
  \utilde\eta(x) := \begin{cases}
  0 &\hbox{ if }x=0\\
\eta(x)&\hbox{ otherwise}.
\end{cases}
\end{align}
and recall $\tilde\eta$ defined in \eqref{tilde-eta} as the configuration $\eta$ with a particle at the origin. Now define
\begin{align}
\label{tilde-sigma-xi}
  (\utilde\sigma_t,\tilde\xi_t):= (\utilde\eta^{\lambda}_t,\tilde\eta^{\rho}_t-\utilde\eta^\lambda_t).
\end{align}
The initial configuration for this process is identical to $(\sigma,\xi)$ out of the origin while at the origin there is a second class particle: $\sigma(0)=0$ and $\xi(0)=1$.

\begin{proposition} Take $\lambda<\rho$ and let $Y^{\lambda,\rho}_t$ be the position of the tagged $\xi$ particle for the process \eqref{tilde-sigma-xi}, initially located at the origin, $Y^{\lambda,\rho}_0=0$. Then, 
\begin{align}
\label{llnxt}
  \lim_{t\to\infty} \frac{Y^{\lambda,\rho}_t}{t} = 1-\lambda-\rho,\qquad \hbox{a.s.}
\end{align}
\end{proposition}

\begin{proof} Denote $G_{y_t}(t)[(\utilde\sigma,\tilde\xi),\omega]$ the flux of $\tilde\xi$ particles along a trajectory $y_t$ for the process $(\utilde\sigma_t,\tilde\xi_t)$. This flux is the difference of $\tilde\eta^\rho$ particle flux and the $\utilde\eta^\lambda$ particle flux:
\begin{align}
G_{y_t}(t)[(\utilde\sigma,\tilde\xi),\omega]&=F_{y_t}(t)[\utilde\eta^{\rho},\omega]-F_{y_t}(t)[\tilde\eta^{\lambda},\omega]\label{a877}\\
&=F_{y_t}(t)[\eta^{\rho},\omega]-F_{y_t}(t)[\eta^{\lambda},\omega]+O(1),\label{319}
\end{align}
where the error $O(1)$ comes from \eqref{flux}. Taking $y_t=at$ for some real number $a$, by the law of large numbers~\eqref{fl1},
\begin{align}
  \limt \frac{G_{at}(t)[(\utilde\sigma,\tilde\xi),\omega]}{ t} &=[\rho(1-\rho)-\lambda(1-\lambda)]-a(\rho-\lambda),\quad \hbox{a.s.}
\end{align}
The limit is negative for $a>1-\lambda-\rho$ and positive for $a<1-\lambda-\rho$.
On the other hand, $G_{at}(t)$ is non increasing in $a$
and, by exclusion, the flux of $\tilde\xi$ particles along $Y^{\lambda,\rho}_t$ is null: $G_{Y^{\lambda,\rho}_t}(t)\equiv 0$. This implies \eqref{llnxt}. 
\end{proof}



\paragraph{Isolated second class particle}
Take $\alpha\in(0,1)$. To create a second class particle for the configuration $\eta^\alpha$ we consider the coupling
\begin{align}
\label{tilde-alpha}
  (\utilde\eta^\alpha_t,\tilde\eta^\alpha_t-\utilde\eta^\alpha_t)
\end{align}
and call
\begin{align}
\label{tilde-alpha1}
   R^\alpha_t:= \{x: \tilde\eta^\alpha_t(x)\neq\utilde\eta^\alpha_t(x)\},
\end{align}
the position at time $t$ of the second class particle in the coupling \eqref{tilde-alpha}. 
\begin{proposition}
\label{isolated}
  We have
\begin{align}
\label{llnrt}
  \lim_{t\to\infty} \frac{R^\alpha_t}{t} = 1-2\alpha,\qquad \hbox{a.s.}
\end{align}
\end{proposition}

\begin{proof} Take $\alpha<\rho$ and consider the coupling
\begin{align}
\label{tilde-sigma-xi-1}
 (\utilde\eta^{\alpha}_t,\tilde\eta^{\rho}_t-\utilde\eta^\alpha_t)
\end{align}
and, as before, denote $Y^{\alpha,\rho}_t$ the position of the tagged second class particle initially at the origin for this process. Recalling that we are using the same $U$ and $\omega$ in the couplings \eqref{tilde-alpha} and \eqref{tilde-sigma-xi-1} we see that both $R^\alpha_t$ and $Y^{\alpha,\rho}_t$ see the same first class particles $\utilde\eta^\alpha_t$ but while $R^\alpha_t$ sees no other particle, $Y^{\alpha,\rho}_t$ is blocked by the second class particles $(\tilde\eta^{\rho}_t-\utilde\eta^\alpha_t)$ to its right. For this reason,
\begin{align}
\label{r>y}
    R^\alpha_t\ge Y^{\alpha,\rho}_t, \qquad \hbox{ if } \alpha<\rho.
\end{align}
On the other hand, take $\lambda<\alpha$ and consider the coupling
\begin{align}
\label{tilde-sigma-xi-2}
 (\utilde\eta^{\lambda}_t,\tilde\eta^{\alpha}_t-\utilde\eta^\lambda_t).
\end{align}
The first class particles for $Y^{\lambda,\alpha}_t$ are $\eta^{\lambda,\alpha}_t\le \eta^\alpha_t$, the first class particles for $R^\alpha_t$. See \eqref{er233} and \eqref{b300} below for more details. Hence
\begin{align}
  \label{r<y}
R^\alpha_t\le Y^{\lambda,\alpha}_t,\qquad \hbox{ if } \lambda < \alpha.
\end{align}
Use the law of large numbers \eqref{llnxt} to conclude. 
\end{proof}

\section{Proof of hydrodynamics: increasing shock}
\label{seven}
In this section we prove Theorem \ref{t4.2} in the shock case $\lambda<\rho$. Recall that in this case the solution $u(r,t) = u^{\lambda,\rho}(r-(1-\lambda-\rho)t)$ is a translation of the initial condition. 

Let  $\Gamma_z:\{0,1\}^\Z\to\{0,1\}^\Z$ be the \emph{cut operator} defined by 
\begin{align}
\label{Gamma}
  \Gamma_z\eta(x) := \eta(x)\one\{x\ge z\}.
\end{align}
This operator, when applied to the configuration $\eta$ cuts the $\eta$-particles to the left of $z$. 
The operator $\Gamma_0$, when applied to the second class particles $\xi$ commutes with the dynamics in the following sense. 
If $\xi(0)=1$ and $Y_t$ is the position of the $\xi$ particle initially at the origin, then
\begin{align}
  \label{gamma}
(\sigma_t[(\sigma,\xi),\omega],\Gamma_{Y_t}\xi_t[(\sigma,\xi),\omega])= (\sigma_t[(\sigma,\Gamma_0\xi),\omega],\xi_t[(\sigma,\Gamma_0\xi),\omega]).
\end{align}
That is, to cut the initial $\xi$ configuration to the left of the origin and evolve until time $t$ is the same as to cut the $\xi_t$ configuration to the left of $Y_t$. The reason is that the initial $\xi$ particles to the left of $Y_0$ are not felt neither by the $\sigma$ particles nor by the $\xi$ particles at $Y_0$ and to the right of $Y_0$, so it is the same to cut them at time 0 than to cut them at time~$t$. Since those particles occupy sites to the left of $Y_t$ at that time, we get \eqref{gamma}.

Let $(\sigma,\xi)$ be a two-class configuration and let
\begin{align}
  \eta:= \sigma + \Gamma_0\xi.
\end{align}
Add a second class particle with respect to $\eta_t$ at the origin at time zero; call $R_t$ its position at time $t$. Add a $\xi$ particle at the origin at time zero; call $Y_t$ its position at time $t$. Then, using \eqref{gamma},
\begin{align}
    R_t&=Y_t \label{er233}\\
  (\utilde\eta_t,R_t)&= (\utilde\sigma_t+ \Gamma_{Y_t}\tilde\xi_t,Y_t). \label{er234}
\end{align}
Recall $\eta^\rho$ and $\eta^{\lambda,\rho}$ are already defined as functions of $U$ and that their different tilded versions are also defined in \eqref{tilde-eta} and \eqref{utilde}. Set $\sigma=\eta^\lambda$ and $\xi=\eta^\rho-\eta^\lambda$. From those definitions we have
\begin{align}
  (\utilde\sigma,\tilde\xi)&= (\utilde\eta^\lambda, \tilde\eta^\rho-\utilde\eta^\lambda)\nn\\
  \utilde\eta^{\lambda,\rho}&= \utilde\sigma+ \Gamma_0\tilde\xi\nn
\end{align}
Let $R^{\lambda,\rho}_t$ be a second class particle with respect to $\utilde\eta^{\lambda,\rho}_t$ and $Y^{\lambda,\rho}_t$ be a $\tilde\xi$ tagged particle for $(\utilde\sigma_t,\tilde\xi_t)$ with $R_0^{\lambda,\rho}=Y_0^{\lambda,\rho}=0$. Then,
\begin{align}
  \label{b300}
  R^{\lambda,\rho}_t&= Y^{\lambda,\rho}_t\nn\\
  (\utilde\eta^{\lambda,\rho}_t,R^{\lambda,\rho}_t)&= (\utilde\sigma_t+ \Gamma_{Y^{\lambda,\rho}_t}\tilde\xi_t,Y^{\lambda,\rho}_t),
\end{align}
for all $t\ge0$ by \eqref{er233}-\eqref{er234}. Roughly speaking, to obtain the system with shock initial condition $\eta^{\lambda,\rho}_t$ and a second class particle $R^{\lambda,\rho}_t$ one can take the system of two classes $(\sigma_t,\xi_t)$ with the right marginals, cut the second class particles to the left of the tagged second class particle $Y^{\lambda,\rho}_t$ and forget the classes for the remaining particles.  
Notice that 
\begin{align}
  \eta^{\lambda,\rho}_t(x) = 
  \begin{cases}
\label{52a}
    \utilde\eta^{\lambda,\rho}_t(x) &\hbox{if } x\neq R^{\lambda,\rho}_t\\[2mm]
\eta^{\lambda,\rho}_t(0)&\hbox{if } x= R^{\lambda,\rho}_t
  \end{cases}
\end{align}


\paragraph{Proof of local equilibrium \eqref{local-equilibrium1} for $\lambda<\rho$}
 In this case \eqref{local-equilibrium1} reduces to 
 \begin{align}
\label{local-equilibrium-t}
   \limt E f_A(  \tau_{rt}\eta^{\lambda,\rho}_{t}) =
   \begin{cases}
     \rho^{|A|} &\hbox{if } r>1-\rho-\lambda\\
     \lambda^{|A|} &\hbox{if } r<1-\rho-\lambda
   \end{cases}  
 \end{align}
Take first $r>(1-\lambda-\rho)$ and denote $Y_t=Y^{\lambda,\rho}_t$ the position of the tagged $\xi$ particle. By \eqref{b300} and \eqref{52a} we get
\begin{align}
 E f_A(  \tau_{rt}\utilde\eta^{\lambda,\rho}_{t})& = E f_A( \tau_{rt}(\utilde\sigma_{t}+\Gamma_{Y_t}\tilde\xi_{t}))\label{s68}\\
&= E\big[f_A(\tau_{rt}(\sigma_{t}+\xi_{t}))\, \one\{Y_t<rt+\min A\}\big]\label{s683}\\
&\qquad +\, E\big[f_A(\tau_{rt}(\utilde\sigma_{t}+\Gamma_{Y_t}\tilde\xi_{t}))\, \one\{Y_t\ge rt +\min A\}\big]\nonumber\\
&
                                                                                      \mathop{\to}_{t\to\infty}\;\rho^{|A|},\label{s71}
\end{align}
where in \eqref{s68} we used \eqref{b300} to get an expression in terms of $(\sigma,\xi)$ and in \eqref{s683} we used the definition of the cut operator $\Gamma$ to erase it and \eqref{52a} to erase the tildes. Since  $Y_t/t\to 1-\lambda-\rho$ a.s., the indicator functions converge a.s.\/ to 1 and zero respectively. Since  $|f_A|\le 1$, the second summand goes to zero and since $\sigma_t+\xi_t=\eta^\rho$ whose law is shift invariant, the first summand converges to $\rho^{|A|}$; this justifies \eqref{s71} and concludes the proof of  \eqref{local-equilibrium-t} when $r>(1-\lambda-\rho)$. 
The same argument shows \eqref{local-equilibrium-t} when $r<(1-\lambda-\rho)$.

\paragraph{\bf Proof of convergence of the density fields}

We use the same argument and notation as in the previous proof. Fix $1-\lambda-\rho<a<b$ and write 
\begin{align}
 \frac1t\sum_{at\le x\le bt}\eta^{\lambda,\rho}_t(x) 
& = \frac1t\sum_{at\le x\le bt} (\sigma_t(x)+\Gamma_{Y_t}\xi_t(x))\nonumber\\
& = \frac1t\sum_{at\le x\le bt} (\sigma_t(x)+\xi_t(x))\,\one\{Y_t<at\}  + (b-a)O(1)\,\one\{Y_t\ge at\}\nonumber\\
& = \frac1t\sum_{at\le x\le bt} \eta^\rho_t(x)\,(1- \one\{Y_t\ge at\})  + (b-a)O(1)\,\one\{Y_t\ge at\}\label{**}\\
& = \frac1t\sum_{at\le x\le bt} \eta^\rho_t(x)  + 2\,(b-a)O(1)\,\one\{Y_t\ge at\}\nonumber\\
& \mathop{\to}_{t\to\infty} \rho (b-a)\label{***}
\end{align}
where \eqref{**} follows from \eqref{b300} and the limit \eqref{***} follows from the law of large numbers for  $\eta^\rho_t\sim\eta^\rho$ and the law of large numbers for the tagged second class particle $Y_t/t\to 1-\rho-\lambda$. 
The same argument applied to $c<d<1-\lambda-\rho$ shows 
\begin{align}
\frac1t\sum_{ct\le x\le dt}\eta^{\lambda,\rho}_t(x) 
& = \frac1t\sum_{ct\le x\le dt} (\sigma_t(x)+\Gamma_{Y_t}\xi_t(x)) \nonumber\\
& = \frac1t\sum_{ct\le x\le dt} \eta^\lambda_t(x) + (d-c)O(1)\,\one\{Y_t\le dt\}\nonumber\\
&
\mathop{\to}_{t\to\infty} \lambda (d-c)
\end{align}
using the law of large numbers for $\eta^\lambda_t$. 
For $d<1-\lambda-\rho<a$ we have 
\begin{align}
0\le \frac1t\sum_{dt\le x\le at}\eta^{\lambda,\rho}_t(x) \le  a-d.\nn
\end{align}
Taking $a,d\to 0$ we conclude that for $c<1-\lambda-\rho<b$ we have 
\begin{align}
  \frac1t\sum_{ct\le x\le dt}\eta^{\lambda,\rho}_t(x) \mathop{\to}_{t\to\infty} \lambda (1-\lambda-\rho-c) +\rho (b-1-\lambda-\rho). 
\end{align}
which is \eqref{aib1} in this case.

\section{Proof of hydrodynamics: rarefaction fan}
\label{eight}
Here we consider $\lambda>\rho$, when the  solution is the rarefaction fan \eqref{rarefaction-solution}.
An essential component of this proof is the law of large numbers for a second class particle Proposition \ref{isolated}.
We first prove a crucial lemma. Recall that the processes $\eta^\rho_t$ and $\eta^{\lambda,\rho}_t$  defined in \eqref{etalambda} and \eqref{etarl1} and $R^\alpha_t$ defined in \eqref{tilde-alpha1} are constructed with the same $U$ and $\omega$ for all $\lambda,\rho,\alpha$.
\begin{lemma}
\label{dom1}
Take $\lambda>\rho$ and 
for each $\alpha\in[0,1]$ let $R^\alpha_t$ be a second class particle initially at the origin for the process $\eta^\alpha_t$ as defined in \eqref{tilde-alpha1}.  Then
\begin{align}
\label{a382c}
  \eta^{\lambda,\rho}_t(x) =
  \begin{cases}
   \eta^\rho_t(x) &\hbox{if } x>R^\rho_t\\
 \eta^\lambda_t(x) &\hbox{if } x<R^\lambda_t.
  \end{cases}
\end{align}
Furthermore, for $\lambda\ge \alpha\ge \rho$ we have
\begin{align}
 \eta^{\lambda,\rho}_t(x) \le \eta^\alpha_t(x),\quad\hbox{for } x> R^\alpha_t, \label{a382b}\\
 \eta^\alpha_t(x) \le  \eta^{\lambda,\rho}_t(x),\quad\hbox{for } x< R^\alpha_t.\label{a382a}
\end{align}
\end{lemma}
\begin{figure}[h]
\centering
\includegraphics[
clip=true, 
scale=.8
]
{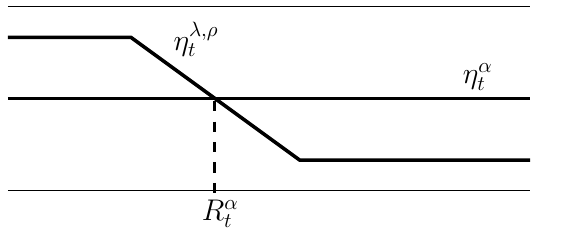} 
\caption{Macroscopic schema of \eqref{a382b} and \eqref{a382a}. The configuration $\eta^{\lambda,\rho}_t$ dominates $\eta^\alpha_t$ to the left of $R^\alpha_t$ and the opposite happens to its right.}
\label{fig8-2}
\end{figure}%
\begin{proof} Consider the process with only one second class particle and constant density $\rho$ given by
\begin{align}
  (\utilde\eta^\rho_t, \tilde\eta^{\rho}_t - \utilde\eta^\rho_t)
\end{align}
and let $R^\rho_t$ be the second class particle for this coupling. On the other hand, define
\[
(\sigma_t,\xi_t) := (\eta^\rho_t, \eta^{\lambda,\rho}_t - \eta^\rho_t),
\]
where $\sigma_t$ are first class particles and $\xi_t$ are second class particles. 

The first identity in \eqref{a382c} is equivalent to 
\begin{align}
  \xi_t(x)=0,\quad\hbox{ for } x>R^\rho_t.
\end{align}
This clearly holds at time $0$ because $R^\rho_0=0$ and $\xi(x) =\eta^{\lambda,\rho}(x)-\eta^\rho(x)=0$ for all $x>0$, by definition. 
Furthermore, $\xi$ particles cannot overpass $R^{\rho}_t$:
\begin{align}
\label{rost14}
 Y_t:= \max\{y:\xi_t(y)=1\}\le R^\rho_t.
\end{align}
\begin{figure}[h]
\centering
\includegraphics[
clip=true, 
scale=.8
]
{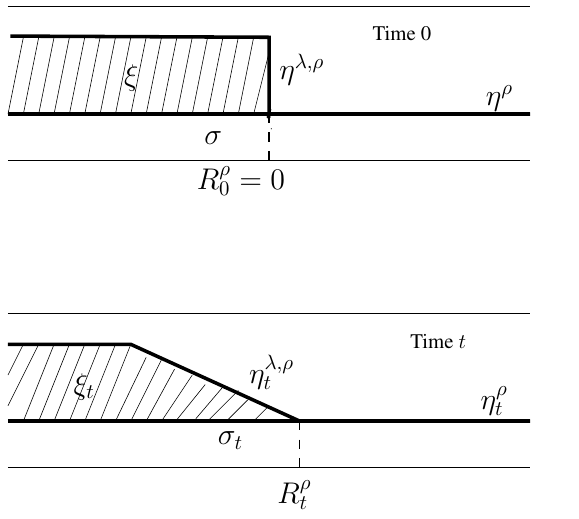} 
\caption{Macroscopic schema of the coupling to show \eqref{a382c}. There are no $\xi_t$ particles to the right of $R^\rho_t$ at time $t$.}
\label{fig81}
\end{figure}%
Let's show \eqref{rost14}. Since $\xi$ particles interact by exclusion among them, we have that the rightmost $\xi$ particle does not feel the $\xi$ particles to its left and hence $Y_t$ behaves as a second class particle for $\eta^\rho_t$, but with a random initial position $Y_0:=\max \{y\le 0: \xi_0(y)=1\}\le 0=R^\rho_0$. One is tempted to say that 2 second class particles with respect to $\eta^\rho_t$ can not overpass, but since we have a precise definition of $Y_0$ (in function of $\eta^\lambda$ and $\eta^\rho$) and $R^\rho_0$ (in function of $\eta^\rho$ and its tilded versions),  we have to explore the following three cases. (a) If $\eta^\rho(0)=0$ and $\eta^\lambda(0)=1$, then $Y_0=0=R^\rho_0$ and both particles will coincide at future times. (b) If $\eta^\rho(0)=\eta^\lambda(0)=1$, then  $Y_0<R^\rho_0$ and $Y_t< R^\rho_t$ for all times because $\sigma_t(R^\rho_t)=1$ and $Y_t$ cannot jump over $\sigma$ particles. (c) If $\eta^\rho(0)=\eta^\lambda(0)=0$, then $Y_0<R^\rho_t$ and $Y_t\le R^\rho_t$ for all times because if there is an arrow at $x$ at time $t$ and $Y_{t-}=x$, $R^\rho_{t-}=x+1$, then after the arrow the particles coalesce $Y_t=R^\rho_t=x+1$ and then continue together for ever. See the following tables for the (b) and (c) cases, the bold numbers correspond to the particles and holes involved in the definition of $Y_t$ or $R^\rho_t$. For instance the first row of case (b) means $\eta^\lambda_t(Y_t) =1$, $\eta^\lambda_t(R^\rho_t) =1$ and the second row of case (c) means  $\tilde\eta^\rho_t(Y_t) =0$, $\tilde\eta^\rho_t(R^\rho_t) =1$. More concisely, in case (b)  $\eta^\rho(0)=1$ and $R^\rho_t$ behaves as a first class particle for $Y_t$ so they exclude each other while in case (c)  $\eta^\lambda(0)=0$ and $R^\rho_t$ behaves as a hole for $Y_t$ so they can coalesce. 
\begin{align}
 (b)  \begin{array}{c|c|ccccc}
    &&&Y&R&\\
    \hline
   &\eta^\lambda&&\bold 1&1&\\[1mm]
   &\eta^\rho&&\mathbf 0& 1&\\[1mm]
   &\tilde\eta^\rho&& 0&\mathbf 1&\\[1mm]
    &\utilde\eta^\rho&&0&\mathbf 0&\\[1mm]
    \hline
      \end{array}
  \qquad(c)  \begin{array}{c|c|ccccc}
    &&&Y&R&\\
    \hline
   &\eta^\lambda&&\mathbf 1&0&\\[1mm]
    &\eta^\rho&&\mathbf 0&0&\\[1mm]
   &\tilde\eta^\rho&&1&\mathbf 1&\\[1mm]
   &\utilde\eta^\rho&& 0&\mathbf 0&\\[1mm]
    \hline
\end{array}
\end{align}


The first identity in \eqref{a382c} follows from \eqref{rost14}.
To get the second identity in \eqref{a382c} define
\[
(\sigma_t,\xi_t) := (\eta^{\lambda,\rho}_t,\eta^\lambda_t - \eta^{\lambda,\rho}_t)
\]
and use an argument analogous to the proof of  \eqref{rost14} to show that
\[
R^\lambda_t\ge \min\{y:\xi_t(y)=1\},
\]
that is,  $\xi_t(x)=0$ for $x< R^\lambda_t$. 

To show   \eqref{a382b} and \eqref{a382a} recall $\lambda\ge\alpha\ge\rho$ and observe that
\begin{align}
  \eta^{\lambda,\rho}_t(x)\le \eta^{\lambda,\alpha}_t(x) = \eta_t^\alpha(x),\qquad\hbox{for } x>R^\alpha_t \\
\eta^\alpha_t(x) = \eta^{\alpha,\rho}_t(x) \le \eta^{\lambda,\rho}_t(x),\qquad\hbox{for } x<R^\alpha_t,
\end{align}
where the inequalities hold by attractivity and the identities are \eqref{a382c}. 
\end{proof}

\begin{corollary}
  Let $\lambda\ge \alpha>\beta\ge \rho$. Then,
\begin{align}
P\Bigl(\liminf_{t\to\infty} \frac1t\sum_x\eta^{\lambda,\rho}_t(x) \one\{x\in ((1-2\alpha)t,(1-2\beta)t)\} &\ge 2(\alpha-\beta)\beta\Bigr)=1 \label{a383b}\\
P\Bigl(\limsup_{t\to\infty} \frac1t\sum_x\eta^{\lambda,\rho}_t(x) \one\{x\in ((1-2\alpha)t,(1-2\beta)t)\} &\le 2(\alpha-\beta)\alpha \Bigr)=1 \label{a383a}
\end{align}
\end{corollary}

\begin{proof}
From \eqref{a382a},
\begin{align}
  \label{a388y}
  \sum_x\eta^\beta_t(x) \one\{x\in ( R^\beta_t,R^\alpha_t)\}\le \sum_x\eta^{\lambda,\rho}_t(x) \one\{x\in ( R^\beta_t,R^\alpha_t)\} \le \sum_x\eta^\alpha_t(x) \one\{x\in ( R^\beta_t,R^\alpha_t)\}.
\end{align}
\begin{figure}[h]
\centering
\includegraphics[
scale=.8
]
{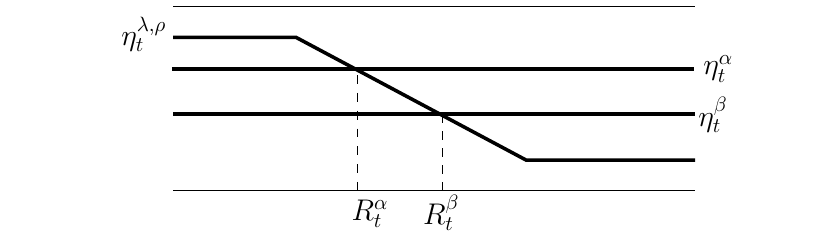}
\caption{Macroscopic schema of \eqref{a388y}}
\label{fig8}
\end{figure}%
From the first inequality in \eqref{a388y},
\begin{align}
  &\sum_x\eta^\beta_t(x) \one\{x\in ((1-2\beta)t,(1-2\alpha)t)\}\nonumber\\
&\qquad\le \sum_x\eta^{\lambda,\rho}_t(x) \one\{x\in ((1-2\beta)t,(1-2\alpha)t)\}\nonumber\\
&\qquad\quad + 2|R^\beta_t-(1-2\beta)t|+2|R^\alpha_t-(1-2\alpha)t)|
\end{align}
Divide by $t$, take $t\to\infty$ and use the law of large numbers for $\eta^\beta_t\sim\eta^\beta$ and for $R^\alpha_t,R^\beta_t$ to get \eqref{a383b}. The same argument using the second inequality in \eqref{a388y} shows \eqref{a383a}. 
\end{proof}

\paragraph{Proof of convergence of the density fields} 
Fix $r\in(1-2\lambda,1-2\rho)$ and use the bound~\eqref{a383a} with $\alpha=k/n$ and $\beta=(k-1)/n$ to obtain
\begin{align}
 &\limsup_t \frac1t\sum_{x\in(rt,(1-2\rho)t)}\eta^{\lambda,\rho}_t(x)\nonumber\\
& \qquad=\; \limsup_t \frac1t\sum_{k=1}^n \sum_x \eta^{\lambda,\rho}_t(x)\, \one\big\{x\in [\hbox{$t(1-2\frac{k}n),t(1-2\frac{k-1}n)$}]\cap[rt,(1-2\rho)t)\big\} \nonumber\\
&\qquad \le\;\sum_{k=1}^n  \frac{k}n \frac2n\, \one\big\{\hbox{$\rho\le \frac{k}n$}\le \hbox{$\frac{1-r}2$}\big\} 
\nonumber \\
&\qquad\mathop{\longrightarrow}_{n\to\infty}\; \int_\rho^{\frac{1-r}2} 2r'dr' \;=\; \Bigl(\frac{1-r}2\Bigr)^2-\rho^2 \;=\; \int_r^{1-2\rho} u(r',1) dr'.\nonumber
\end{align}
The same argument using \eqref{a383b} shows that
\begin{align}
 \liminf_t \frac1t\sum_{x\in (rt,(1-2\rho)t)}\eta^{\lambda,\rho}_t(x) &  \ge \int_r^{1-2\rho} u(r',1) dr'.\nonumber
\end{align}
This proves \eqref{aib} for intervals $(a,b)\subset (1-2\lambda,1-2\rho)$. Take now $a<1-2\lambda$ and use the second identity in \eqref{a382c} and the law of large numbers for $R^\lambda_t$ to conclude that 
\begin{align}
\label{df3}
 \lim_t \frac1t\sum_{x\in (at,(1-2\lambda)t)}\eta^{\lambda,\rho}_t(x) & = \lambda(1-2\lambda-a)=  \int_a^{1-2\lambda} u(r',1) dr'.
\end{align}
Take $b>1-2\rho$ and use the first identity in \eqref{a382c} and the law of large numbers for $R^\rho_t$ to conclude
\begin{align}
\label{df4}
 \lim_t \frac1t\sum_{x\in ((1-2\rho)t, bt)}\eta^{\lambda,\rho}_t(x) &= \rho(b-(1-2\rho))=  \int_{1-2\rho}^b u(r',1) dr'.\qquad 
\end{align}

\paragraph{Proof of density profile and local equilibrium} Take a finite integer set $A$ and recall $f_A(\eta) = \prod_{x\in A}\eta(x)$. Take  $\lambda\ge \alpha>\beta\ge \rho$. From \eqref{a382b}-\eqref{a382a} we have
\begin{align}
 B_t:=\big\{R^\alpha_t<rt+x< R^\beta_t, x\in A\big\}\subset \Bigl\{f_A(\tau_{rt}\eta^\alpha_t) \ge
  f_A(\tau_{rt}\eta^{\lambda,\rho}_t) \ge f_A(\tau_{rt}\eta^\beta_t)
  \Bigr\}.  \label{a385}
\end{align}

Hence, denoting $\one B$ the indicator function of the set $B$, we have
\[
E(f_A(\tau_{rt}\eta^\beta_t)\,\one B_t) \le  E(f_A(\tau_{rt}\eta^{\lambda,\rho}_t)\,\one B_t) \le E(f_A(\tau_{rt}\eta^\alpha_t)\,\one B_t).
\]
By the law of large numbers for $R^\alpha_t$ and $R^\beta_t$, for $r\in ((1-2\alpha),(1-2\beta))$ we have
$ \lim_t \one B_t=1$ a.s..
Hence, since $|f_A|\le 1$, for $r\in ((1-2\alpha),(1-2\beta))$, 
\[
\beta^{|A|} \le \liminf_t E(f_A(\tau_{rt}\eta^{\lambda,\rho}_t))\le \limsup_t E(f_A(\tau_{rt}\eta^{\lambda,\rho}_t)) \le \alpha^{|A|}.
\]
Take $\alpha\searrow \frac{1-r}2$ and $\beta\nearrow \frac{1-r}2$ to get
\[
\lim_t E(f_A(\tau_{rt}\eta^{\lambda,\rho}_t))=  \Bigl(\frac{1-r}2\Bigr)^{|A|}= u(r,1)^{|A|}.
\]
This proves local equilibrium for $r$ in the rarefaction fan $((1-2\lambda),(1-2\rho))$. For $r\ge 1-2\rho$ we know that $\eta^{\lambda,\rho}_t(x) = \eta^\rho_t(x)$ when $x>R^\rho_t$. This together with the law of large numbers for $R^\rho_t$ allows to conclude. The same argument holds for $r<1-2\lambda$.

\section{Notes and references}
\label{notes}
There are many papers about hydrodynamics of interacting particles systems. We just quote some reviews and books. De Masi and Presutti   \cite{MR1175626,MR757003}, Kipnis and Landim \cite{MR1707314} and Lebowitz, Presutti and Spohn \cite{MR971034}.

Lax (1972) shows the role of characteristics to solve the initial value problem of the Burgers equation. See also Evans \cite{MR1625845}.
Rezakhanlou \cite{MR1298939}  shows there that if the initial condition presents no decreasing discontinuity at $a$, then there is only one characteristic emanating from $a$. Rezakhanlou \cite{MR1326665} shows that a local perturbation of the initial condition of the Burgers equation behaves like the characteristics or a shock. 

The convergence of the hydrodynamic limit of the tasep to the Burgers equation has also different approaches and results. 
The local-equilibrium convergence \eqref{local-equilibrium} was proven by Liggett \cite{MR0410986,MR0445644} for the case
$r = 0$, before the connection between the process and the Burgers equation appeared. The first paper realizing this connection was Rost \cite{MR635270} who studied the rarefaction fan case. Rost uses the sub-additive ergodic theorem to show almost sure convergence of the density fields and then a comparison with stationary systems of queues to identify the limit and to show local equilibrium; see also Liggett's book \cite{MR2108619}. The result is generalized by Seppalainen \cite{MR1625007,MR1677061,MR1681094, timo-exclusion-notes}, who uses it to prove almost sure convergence of density fields for a large class of initial conditions. Proofs for more general initial profiles were provided by Benassi and Fouque \cite{MR885130}, Benassi, Fouque, Saada, and Vares \cite{MR1115810}. Andjel and Vares \cite{MR892931} prove convergence of the expectation of the density fields for general initial profiles for a class of processes including the tasep, without using subadditivity. Andjel, Ferrari and Siqueira \cite{MR2087959} extended the arguments of Rost to the case of non-nearest neighbor jumps. 

In dimension $d\ge 1$, Rezakhanlou \cite{MR1130693} proves convergence in probability of the density fields while Landim \cite{MR1245290} shows that this limit is enough to have local equilibrium. See also Landim \cite{MR1127714,MR1141248}. Recent strong approach to hydrodynamics without subadditivity can be found in the work of Bahadoran, Guiol, Ravishankar, and Saada \cite{MR2257649}, \cite{MR2578381}.

Ferrari, Kipnis and Saada \cite{MR1085334} and the author \cite{MR1142763,MR1263704} used the laws of large numbers for tagged and second class particles to show hydrodynamics in the shock case. The structure of Sections \ref{six} and \ref{seven} follows the survey \cite{MR1263704} but with an important simplification. Here we only use a law of large numbers for the tagged particle and attractive couplings to obtain all the other results while the arguments in \cite{MR1263704} also used asymptotic properties of the invariant measure for first and second class particles, a more refined property. 

\paragraph{Further results not discussed in this paper} Local equilibrium does not hold at the discontinuity points of the solution $u$. Wick \cite{MR802566}, Andjel, Bramson and Liggett \cite{MR945111}, De Masi, Kipnis, Presutti, and Saada \cite{MR1008228} have proven partial results. The author \cite{MR1263704} proved that the limit is a convex combination of the product measures with densities $\lambda$ and $\rho$, depending if the second class particle for $\eta^{\lambda,\rho}_t$ is to the right or left of $(1-\lambda-\rho)t$. 

Microscopic interfaces. A second class particle with respect to a product initial configuration with densities $\lambda<\rho$ to the left and right of the origin, respectively, sees at any time $t$ a measure that is absolutely continuous with respect to the product measure with a bounded Radom-Nikodim derivative. In fact, there exists an invariant measure for the process as seen from the second class particle which is absolutely continuous with respect to the product measure. This started with \cite{MR1085334,MR1142763,MR1263704}, then   Derrida, Janowsky, Lebowitz and Speer \cite{MR1251221} computed the measure, from where subsequent progress done by Ferrari Fontes Kohayakawa \cite{MR1298099} and  Angel \cite{MR2216456} permitted Ferrari and Martin \cite{MR2319708} to give a complete description of that measure in terms of the output of a discrete-time stationary MM1 queue. 

Diffusive fluctuations. 
The flux or current of particles along lines different from the characteristic have variance of order $t$ explicitly computed by Ferrari and Fontes \cite{MR1288133}, see also Ben Arous and Corwin \cite{MR2778798}.  For the second class particle in the shock also has variance of order $t$, computed in \cite{MR1278887,MR1257605}.

The flux of particles along a characteristic has non-diffusive fluctuations, while a second class particle in a translation invariant Bernoulli measure has super diffusive behavior \cite{MR1288133}. Ferrari and Spohn \cite{MR2217295} compute the equilibrium current fluctuations along the characteristic of order $t^{1/3}$ and show that the limit in distribution converges to the GUE Tracy-Widom distribution. For the growth process associated to the tasep Johansson \cite{MR1737991} computes limiting fluctuations of order $t^{1/3}$, and find the limit distribution, see  Prahoffer and Spohn \cite{MR1901953} and Ben Arous and Corwin \cite{MR2778798}.  Balasz, Cator and Seppalainen \cite{MR2268539} compute the order $t^{2/3}$ for the variance of the mentioned growth model.  

The second class particle in the rarefaction fan converges almost surely to a uniform random variable in $[-1,1]$. See Ferrari and Kipnis \cite{MR1340034} for convergence in distribution and Mountford and Guiol \cite{MR2134103} Ferrari, Pimentel and Martin~\cite{MR2150188,MR2498679} for a.s. convergence. Further results can be found in Ferrari, Gonçalves and Martin \cite{MR2572163} and Amir, Angel and Valko~\cite{MR2857238}.

\section*{Acknowledgments} 

I thank an anonymous referee for his extremely careful reading and useful comments. 
This paper started as a mini course given in the  CIMPA school \emph{Random processes and optimal configurations
in analysis}, Buenos Aires, July 2015, chaired by Jorge Antezana.

\bibliographystyle{abbrv}

\bibliography{bursep-cimpa}

\end{document}